\documentclass[11pt,a4paper]{amsart}

\usepackage[utf8]{inputenc}
\usepackage[T1]{fontenc}
\usepackage[english]{babel}
\usepackage{amsmath,amssymb,amsthm,mathtools,color}
\usepackage{geometry}
\usepackage{hyperref}
\usepackage{enumitem}

\newtheorem{theorem}{Theorem}
\newtheorem{proposition}{Proposition}
\newtheorem{lemma}{Lemma}
\newtheorem{corollary}{Corollary}
\newtheorem{remark}{Remark}

\newcommand{\R}{\mathbb R}

\newcommand{\eps}{\varepsilon}
\newcommand{\cA}{\mathcal A}

\newcommand{\cQstar}{\mathcal Q_*}

\title[Optimizer for quantitative isoperimetric ratio in $\R^3$]{An existence result for a quantitative isoperimetric inequality in $\R^3$ involving the Hausdorff asymmetry}

\author[Bove]{Silvio Bove}

\author[Croce]{Gisella Croce}

\author[Pisante]{Giovanni Pisante}

\address[S. Bove]{Dipartimento di Ingegneria Elettrica e dell'Informazione ``M. Scarano'', Universit\`{a} degli Studi di Cassino e del Lazio Meridionale, Via G. Di Biasio n.~43, 03043 Cassino (FR), Italy}
\email{silvio.bove@unicas.it}

\address[G. Croce]{SAMM, UR 4543 Universit\'e Paris 1 Panth\'eon-Sorbonne, FR 2036 CNRS, 90, rue Tolbiac 75013 France}
\email{gisella.croce@univ-paris1.fr}

\address[G. Pisante]{Dipartimento di Matematica e Fisica, Universit\`a degli Studi della Campania ``Luigi Vanvitelli'', Viale Lincoln 5, 81100 Caserta, Italy}
\email{giovanni.pisante@unicampania.it}

\begin{document}

\begin{abstract}
We study the existence of an optimizer for a quantitative isoperimetric ratio $\mathcal{Q}_*$ in $\mathbb{R}^3$ involving the Hausdorff asymmetry. We prove that $\mathcal{Q}_*$ attains its minimum over the class $\mathcal{A}$ of convex bodies of fixed volume.
\end{abstract}

\maketitle



\section{Introduction}
\label{sec:introduction}
The classical isoperimetric inequality states that among all sets $E \subset \mathbb{R}^n$ with finite Lebesgue measure $|E|$, the Euclidean ball $B$ minimizes the perimeter $P(E)$. This result is expressed by the positivity of the isoperimetric deficit
\begin{equation}\label{eq:isoperimetric}
    \delta(E) := \frac{P(E) - P(B)}{P(B)} \ge 0,
\end{equation}
where $B$ is a ball such that $|B|=|E|$, and equality holds if and only if $E$ coincides with a ball.
Starting from the pioneering work \cite{Bonnesen1924} of Bonnesen in the 1920s for planar convex sets, a central theme in geometric analysis has been the stability of this inequality, namely the problem of quantifying how close a set $E$ is to a ball in terms of the isoperimetric deficit. This leads to quantitative inequalities of the form
\begin{equation}\label{eq:quantitative}
    \delta(E) \ge c \, \alpha(E)^k,
\end{equation}
where $\alpha(E)$ is a scale-invariant asymmetry index measuring the deviation of $E$ from sphericity, and $c, k$ are positive constants depending only on the dimension $n$.

Over the last thirty years, several notions of asymmetry have been employed in the literature to establish inequalities of the form \eqref{eq:quantitative}. The first one, often called the \textit{spherical deviation}, involves the Hausdorff distance and was introduced by Fuglede in \cite{Fuglede1989}:
\begin{equation}
    \lambda_{\mathcal{H}}(K) := d_H(K, B(x_G, r_K)),
\end{equation}
where $d_H$ denotes the Hausdorff distance and $B(x_G, r_K)$ is the ball centered at the barycenter $x_G$ of $K$, with volume equal to $|K|$. Fuglede established a series of stability inequalities for convex sets and nearly spherical sets, later extended in \cite{FGP12} to a larger family of sets.
The resulting estimates depend strongly on the dimension $n$. While a quadratic estimate $\delta(K) \ge C \lambda_{\mathcal{H}}(K)^2$ holds in the plane ($n=2$), Fuglede showed that this is no longer optimal for $n \ge 3$. In particular, in dimension $n=3$ the sharp stability estimate takes the form
\begin{equation}\label{eq:fuglede_est}
    \delta(K) \ge C \frac{\lambda_{\mathcal{H}}(K)^2}{|\log \lambda_{\mathcal{H}}(K)|},
\end{equation}
involving a logarithmic correction term. Hausdorff-type estimates in convex and anisotropic settings are discussed by Esposito, Fusco and Trombetti~\cite{EFT2005}.

To measure the distance of a set from a ball, two $L^1$-type asymmetries have also been considered. The Fraenkel asymmetry (see \cite{Hall1992})
\begin{equation}\label{def-lambda}
\lambda(\Omega)=\inf_{y\in \mathbb{R}^n}\frac{|\Omega \Delta B_y|}{|\Omega|},
\end{equation}
where $B_y$ is the ball centered at $y$ such that $|B_y|=|\Omega|$ and $\Delta$ denotes the symmetric difference of sets. It was introduced by Fraenkel. The barycentric asymmetry
\[
\lambda_0(\Omega)=\frac{|\Omega \Delta B_{x^G}|}{|\Omega|},
\]
was later proposed and studied by Fuglede in \cite{Fuglede93}. Sharp quantitative isoperimetric inequalities of the form \eqref{eq:quantitative} with $k=2$ have been established for these asymmetries in $\R^n$ by several authors (see, for instance, \cite{AFN11}, \cite{FiMP}, \cite{FMP08}, \cite{CL12}, \cite{FGP12}, \cite{DL}, \cite{Fuglede93}, \cite{Gambicchia}).

The question of the existence and the shape of an optimal set for such inequalities, however, remains largely open. The results available in the literature concern, almost exclusively, problems formulated in the plane.
Alvino, Ferone and Nitsch \cite{AFN11} proved the existence of an optimal set for the quantitative isoperimetric inequality with the Fraenkel asymmetry, $\delta \geq C \lambda^2$, within the class of convex sets, and identified its shape as a precisely described stadium.
In the case of the whole class of sets in $\R^2$, the existence of a minimizer has been proved, but the shape is only conjectured (see \cite{CL12} and \cite{bianchini:hal-01181104}).
The existence of a minimizer for $\delta / \lambda_0^2$ has been established within the class of convex sets of the plane (see \cite{BCHAnnali} and \cite{preprint}).
As for the Hausdorff asymmetry, Alvino, Ferone and Nitsch \cite{AFN09} studied $\delta \geq C \lambda_{\mathcal{H}}^2$ within the class of convex sets, exhibiting an optimal set.

In the present paper we consider a similar problem in $\R^3$, with the Hausdorff asymmetry
\[
 \lambda_*(F):=\inf_{x\in\R^3}d_H(F,B+x).
\]
More precisely, we study
\begin{equation}\label{eq:Q-intro}
 \cQstar(F)
 :=
 \frac{
 \delta(F)
 \log\left(\delta(F)+\frac1{\delta(F)}\right)
 P(F)^3
 }{
 \lambda_*(F)^2
 },
\end{equation}
on the class
\[
 \cA:=\{F\subset\R^3:\ F \text{ is a convex body and } |F|=4\pi/3\}.
\]
We set the value of the functional $\cQstar$ to be $+\infty$ on any ball, since both \(\delta(F)\) and \(\lambda_*(F)\) vanish.
The presence of the logarithm in the definition is natural and optimal, as observed by Fuglede~\cite{Fuglede1989}. Indeed, a localized radial displacement of height $\lambda$ on a small spherical cap has a capacitary cost of order $\lambda^2 / |\log\lambda|$, rather than of order $\lambda^2$, as one sees by linearizing the perimeter functional near the sphere (cf.\ Lemmas~\ref{lem:one-disk-logarithmic-estimate},~\ref{lem:spectral}, and~\ref{lem:fuglede-small-slope}). Hence the natural local scale is
\[
 \delta(F)\simeq
 \frac{\lambda_*(F)^2}{|\log\lambda_*(F)|},
\]
and the logarithmic correction is designed to compensate for this loss. This capacitary interpretation is consistent with the classical connection between asymmetry and capacity developed by Hall, Hayman and Weitsman~\cite{HHW1991}.

The perimeter factor $P(F)^3$ plays a different role. It is harmless in the local analysis near the ball, but prevents minimizing sequences from escaping through elongated convex bodies; in other words, it ensures the coercivity of the functional. If $D(F)$ denotes a diameter-type radial size, convexity and the volume constraint yield $P(F) \gtrsim D(F)^{1/2}$ for large $D(F)$. In the same regime $\lambda_*(F)$ is comparable with $D(F)$, while $\delta(F)$ grows like $P(F)$. Thus
\[
 \frac{\delta(F)P(F)^3}{\lambda_*(F)^2}
 \gtrsim
 \frac{P(F)^4}{D(F)^2}
 \gtrsim 1,
\]
and the additional logarithmic factor forces divergence along such sequences. More generally, replacing $P(F)^3$ by $P(F)^q$ would yield the large-diameter scale
\[
 \frac{\delta(F)P(F)^q}{\lambda_*(F)^2}
 \gtrsim D(F)^{(q-3)/2}.
\]
The quotient defining $\cQstar$ thus combines two distinct features: the logarithmic correction captures the local capacitary scale near the ball, while the perimeter factor provides the global coercivity.

Our main result is the existence of a minimizer for the variational problem associated with $\cQstar$ in the class $\cA$. The main difficulty is that a minimizing sequence for $\cQstar$ may fail to converge to an admissible finite-energy minimizer only by approaching the ball, where both $\delta(F)$ and $\lambda_*(F)$ vanish. One therefore needs a lower bound, in this degenerate regime, that is strictly larger than the value of at least one explicit competitor.
For every sequence $(F_j)_j\subset\cA$ such that $\lambda_*(F_j)\to0$, we will prove that
\begin{equation}\label{eq:intro-local-gap}
 \liminf_{j\to\infty}\cQstar(F_j)
 \ge
 \frac{64}{3}\pi^3.
\end{equation}
The constant is not claimed to be optimal: a sharper capacitary analysis may lead to a larger local constant, but the lower bound \eqref{eq:intro-local-gap} is already sufficient to apply the direct methods of the calculus of variations and establish existence. Indeed, the explicit prolate spheroid constructed below satisfies
\[
 \cQstar(F)<368<16\pi^3<\frac{64}{3}\pi^3.
\]
Consequently, a minimizing sequence with values below $16\pi^3$ cannot converge, up to translations, to the family of balls.

\section{Main results}
\label{sec:setting}

In this section we state the
local gap near the family of balls, construct an explicit competitor whose value lies below that gap, and then combine these two
facts to prove the existence of a minimizer for \(\cQstar\) in \(\cA\).

We start by stating in the next theorem the lower bound for $\cQstar$ near the balls. The proof of this result will be discussed in the next section.

\begin{theorem}
\label{thm:rough-local-gap}
Let \((F_j)_j\subset\cA\) be a sequence such that
\[
 \lambda_*(F_j)\to0.
\]
Then
\[
 \liminf_{j\to\infty}\cQstar(F_j)
 \ge
 \frac{64}{3}\pi^3
 >
 16\pi^3.
\]
\end{theorem}

The constant \((64/3)\pi^3\) is not claimed to be sharp. It is the constant
obtained from the present combination of a one-disk capacitary estimate and
the recentered spectral gap. This lower bound is sufficient for the existence
argument, since the explicit prolate spheroid constructed below satisfies
\(\cQstar(F)<16\pi^3\), as we prove in Proposition~\ref{prop:competitor-below-gap} below.

\begin{proposition}
\label{prop:competitor-below-gap}
There exists a set \(F\in\cA\) such that
\begin{equation}\label{eq:F-below-16pi3}
 \inf_{E\in\cA}\cQstar(E)
 \le
 \cQstar(F)
 <16\pi^3.
\end{equation}
\end{proposition}

\begin{proof}
Consider the prolate spheroid
\[
 F:=
 \left\{(x,y,z)\in \mathbb{R}^3 :
 \frac{x^2+y^2}{a^2}+\frac{z^2}{c^2}\le1
 \right\},
\]
where $a=\frac1{\sqrt6}$, $c=6$, so that $a^2c=1$,
and therefore
$|F|=\frac{4\pi}{3}a^2c=\frac{4\pi}{3}$, so that \(F\in\cA\).
We now compute \(\lambda_*(F)\) directly from support functions. The support
function of \(F\) is
\[
        h_F(\nu)=\sqrt{a^2(\nu_1^2+\nu_2^2)+c^2\nu_3^2},
        \qquad \nu\in S^2.
\]
Hence
\[
        d_H(F,B)=\|h_F-1\|_{L^\infty(S^2)}
        =\max\{c-1,1-a\}=5.
\]
Moreover, for every translation vector \(x\in\R^3\), evaluating the support
functions of \(F\) and \(B+x\) in the directions \(e_3\) and \(-e_3\) gives
\[
        d_H(F,B+x)
        \ge
        \max\{|c-1-x_3|,|c-1+x_3|\}
        \ge c-1=5.
\]
Therefore the optimal translation is the origin and
\[
        \lambda_*(F)=d_H(F,B)=5.
\]
The surface area of a prolate spheroid with semiaxes \(a,a,c\),
where \(0<a<c\), is
\begin{equation}\label{eq:prolate-area-formula}
 P(F)=2\pi a^2
 \left(
 1+\frac{c}{ae}\arcsin e
 \right),
 \qquad
 e=\sqrt{1-\frac{a^2}{c^2}}.
\end{equation}
For the above choice of \(a\) and \(c\), formula~\eqref{eq:prolate-area-formula}
gives the explicit numerical bounds
\[
        P(F)<24.23,
        \qquad
        \delta(F)=\frac{P(F)-4\pi}{4\pi}<0.929.
\]
Since \(\lambda_*(F)=5\), and since the function
\(t\mapsto t\log(t+t^{-1})\) is increasing on \((0,1)\), we get
\[
        \cQstar(F)
        <
        \frac{0.929\,\log(0.929+0.929^{-1})\,24.23^3}{25}
        <368<16\pi^3.
\]
This proves the claim.
\end{proof}

Now we are in position to prove the main existence theorem.

\begin{theorem}
\label{thm:existence-clean}
The functional \(\cQstar\) attains its minimum on \(\cA\).
\end{theorem}

\begin{proof}
Let \(F\in\cA\) be the competitor constructed in
Proposition~\ref{prop:competitor-below-gap}. Let \((F_j)_j\subset\cA\) be a
minimizing sequence. Then
\[
        \cQstar(F_j)\to\inf_{E\in\cA}\cQstar(E)
        \le \cQstar(F)<16\pi^3,
\]
and in particular
\begin{equation}\label{eq:Qstar-bounded-sequence}
        \sup_j \cQstar(F_j)<+\infty.
\end{equation}

Set \(\lambda_j:=\lambda_*(F_j)\). Up to translations, choose
\(x_j\in\R^3\) realizing the minimum in the definition of \(\lambda_*\), and
set \(K_j:=F_j-x_j\). Thus
$
        d_H(K_j,B)=\lambda_j$ and $\cQstar(K_j)=\cQstar(F_j).
$
We first show that \((\lambda_j)_j\) is bounded. Suppose, by contradiction,
that \(\lambda_j\to+\infty\) along a subsequence. If \(\lambda_j>1\), then
$
        \operatorname{diam}(K_j)\ge \lambda_j.
$
Indeed, for any point \(p_j\in K_j\), since
\(K_j\subset B_{\operatorname{diam}(K_j)}(p_j)\) and every point of
\(B+p_j\) is at distance at most \(1\) from \(K_j\), one has
$
        d_H(K_j,B+p_j)
        \le \max\{\operatorname{diam}(K_j),1\}.
$
The definition of \(\lambda_*\) then gives the claim.

We now recall that, as a consequence of Schwarz symmetrization, every
convex body \(K\subset\R^3\) with \(|K|=4\pi/3\) satisfies
\begin{equation}\label{eq:diameter-perimeter-lower-bound}
        P(K)\ge \frac{2\pi}{\sqrt3}\,\operatorname{diam}(K)^{1/2}.
\end{equation}
Indeed, let \(\ell=\operatorname{diam}(K)\) and \(R\) be the
concave profile of the Schwarz symmetral of \(K\) with respect to a
diametral axis. Schwarz symmetrization does not increase perimeter, and
\(\displaystyle \int_0^\ell R(z)^2\,dz=\frac 43\) (see Baernstein~\cite[Ch.~2]{Baernstein2019} for the standard properties of
Schwarz symmetrization). Hence
\(\max R\ge \sqrt{4/(3\ell)}\), and the concavity of \(R\) gives
\(\displaystyle \int_0^\ell R(z)\,dz\ge \sqrt{\ell/3}\). Therefore
\[
        P(K) \geq P(K^\sharp)
        \ge 2\pi\int_0^\ell R(z)\,dz
        \ge \frac{2\pi}{\sqrt3}\,\ell^{1/2};
\]
Applying \eqref{eq:diameter-perimeter-lower-bound} to \(K_j\), we get
\[
        P(K_j)\ge c\lambda_j^{1/2},
        \qquad
        \delta(K_j)\ge c\lambda_j^{1/2}
\]
for \(j\) large, with \(c>0\). Hence
\[
        \log\bigl(\delta(K_j)+\delta(K_j)^{-1}\bigr)
        \ge c\log\lambda_j,
\]
and consequently
\[
\begin{aligned}
        \cQstar(F_j)
        &=
        \frac{\delta(K_j)
        \log\bigl(\delta(K_j)+\delta(K_j)^{-1}\bigr)P(K_j)^3}
        {\lambda_j^2}                                      \\
        &\ge
        c\frac{\lambda_j^{1/2}\log\lambda_j\,\lambda_j^{3/2}}
        {\lambda_j^2}
        =c\log\lambda_j\to+\infty.
\end{aligned}
\]
This contradicts~\eqref{eq:Qstar-bounded-sequence}. Therefore
\((\lambda_j)_j\) is bounded and consequently \((K_j)_j\) is uniformly bounded. By
Blaschke's selection theorem and the continuity of volume under Hausdorff
convergence in the convex class, up to a subsequence
$
        K_j\to F_*
$
in Hausdorff distance for some \(F_*\in\cA\).
We claim that \(F_*\) is not a ball. If \(F_*\) were a ball, then
\(\lambda_*(K_j)=\lambda_*(F_j)\to0\). The local gap of
Theorem~\ref{thm:rough-local-gap} would imply
\[
        \liminf_{j\to\infty}\cQstar(F_j)
        \ge \frac{64}{3}\pi^3.
\]
On the other hand, by the minimizing property and by
Proposition~\ref{prop:competitor-below-gap},
\[
        \lim_{j\to\infty}\cQstar(F_j)
        =\inf_{E\in\cA}\cQstar(E)
        \le \cQstar(F)<16\pi^3<\frac{64}{3}\pi^3,
\]
a contradiction.

Since \(F_*\) is not a ball, the standard continuity of intrinsic volumes
under Hausdorff convergence of convex bodies, together with the Lipschitz
continuity of \(\lambda_*\) with respect to the Hausdorff distance, gives
$
        \cQstar(K_j)\to\cQstar(F_*)
$
(see Schneider~\cite[Sec.~4.2]{Schneider2014}). Since
\(\cQstar(K_j)=\cQstar(F_j)\), we conclude that
\[
        \cQstar(F_*)
        =\lim_{j\to\infty}\cQstar(F_j)
        =\inf_{E\in\cA}\cQstar(E).
\]
Thus \(F_*\) is a minimizer.
\end{proof}

\section{Proof of the main lower bound}
In this section, we will give the detailed proof of Theorem~\ref{thm:rough-local-gap}. We begin with two elementary geometric facts which will be used later in the
local analysis near the unit ball. Before stating them, we fix the notation
for balls. If \(x_0\in\R^3\) and \(\rho>0\), we write
\[
        B_\rho(x_0):=\{x\in\R^3:\ |x-x_0|<\rho\}.
\]
When the center is not displayed, the ball is centered at the origin:
\[
        B_\rho:=B_\rho(0).
\]
In particular, throughout the paper \(B=B_1(0)\) denotes the unit ball
centered at the origin.

The first lemma says that, as long as two convex bodies remain uniformly
away from degeneracy and infinity, small Hausdorff distance implies small
uniform distance between their radial functions. This is the elementary
geometric estimate which allows us to pass from Hausdorff control to radial
control.

\begin{lemma}
\label{lem:uniform-radial-stability}
Let \(0<r<R\). There exists a constant \(C=C(r,R)>0\) with the
following property. If \(K,L\subset\R^3\) are convex bodies satisfying
\[
 B_r\subset K,L\subset B_R,
\]
then their radial functions with respect to the origin satisfy
\[
 \|\rho_K-\rho_L\|_{L^\infty(S^2)}
 \le C d_H(K,L).
\]
\end{lemma}

\begin{proof}
Set \(\delta:=d_H(K,L)\). The Hausdorff inclusions give
\[
 K\subset L+\delta B,
 \qquad
 L\subset K+\delta B.
\]
We first recall a uniform Lipschitz bound for radial functions in this
class. We denote by \(L^\circ\) the polar body of \(L\), namely
\[
        L^\circ
        :=
        \{z\in\R^3:\ z\cdot x\le 1 \text{ for every } x\in L\}.
\]
Since \(B_r\subset L\subset B_R\), taking polars reverses the inclusions and
gives
$
        B_{1/R}\subset L^\circ\subset B_{1/r}.
$
Moreover, for every \(\omega\in S^2\),
\[
        \rho_L(\omega)=\frac{1}{h_{L^\circ}(\omega)}.
\]
The support function \(h_{L^\circ}\) is Lipschitz on \(S^2\), with
Lipschitz constant bounded by \(1/r\), and satisfies
\(h_{L^\circ}\ge 1/R\). Hence \(\rho_L\) is Lipschitz on \(S^2\), with a
constant depending only on \(r\) and \(R\). The same bound holds for
\(\rho_K\).

Fix \(\omega\in S^2\), and set \(s:=\rho_K(\omega)\). Then
\(s\omega\in K\subset L+\delta B\). Hence there exists \(y\in L\) such
that \(|s\omega-y|\le\delta\). If \(\delta>r/2\), then the desired bound is
trivial after increasing \(C(r,R)\), since both radial functions take their
values in \([r,R]\). We may therefore assume \(\delta\le r/2\). Writing
\(y=|y|\eta\), we have \(|s-|y||\le\delta\) and \(|y|\ge s-\delta\ge r/2\).
Consequently,
\[
 |\eta-\omega|
 =\left|\frac{y}{|y|}-\omega\right|
 \le C(r)\delta.
\]
Since \(y=|y|\eta\in L\), by the definition of the radial function we have
$
        |y|\le \rho_L(\eta).
$
Using the uniform Lipschitz bound for \(\rho_L\), we therefore get
\[
\begin{aligned}
        \rho_L(\omega)
        &\ge \rho_L(\eta)-C(r,R)|\eta-\omega|  \\
        &\ge |y|-C(r,R)\delta                  \\
        &\ge s-C(r,R)\delta.
\end{aligned}
\]
Thus \(\rho_K(\omega)\le \rho_L(\omega)+C(r,R)\delta\). Interchanging
\(K\) and \(L\) gives the opposite inequality. Taking the supremum over
\(S^2\) proves the claim.
\end{proof}

We next prove a simple recentering lemma. If a convex body is close to the
unit ball, we can translate it by a vector controlled by the Hausdorff distance
so that the radial perturbation has zero first moment. This gives a convenient normalization and keeps the estimates in terms of the
translation-invariant distance \(\lambda_*\).

\begin{lemma}
\label{lem:first-moment-recentering}
There exist \(\eps_0>0\) and \(C>0\) with the following property. Let
\(K\) be a convex body such that \(|K|=4\pi/3\) and
\(d_H(K,B)<\eps_0\). Then there exists a vector \(a(K)\in\R^3\), with
\(|a(K)|\le C d_H(K,B)\), such that, writing \(K^0:=K-a(K)\) and
\(\partial K^0=\{(1+u(\omega))\omega:\omega\in S^2\}\), one has
\begin{equation}\label{eq:first-moment-zero}
 \int_{S^2} u(\omega)\,\omega\,d\sigma(\omega)=0.
\end{equation}
Moreover
\[
 d_H(K^0,B)\le C d_H(K,B).
\]
In particular, \(d_H(K^0,B)\to0\) whenever \(d_H(K,B)\to0\).
\end{lemma}

\begin{remark}
The vector \(a(K)\) is not claimed to be unique and is not required to
minimize the distance from \(K\) to a translated unit ball. Its only role is
to fix a local gauge for nearly spherical sets: after translating by
\(a(K)\), the radial perturbation has zero first spherical moment. This is
the normalization used later to remove the degree-one spherical harmonics.
\end{remark}

\begin{proof}
Set \(\eps:=d_H(K,B)\). If \(\eps=0\), then \(K=B\) and the claim follows
with \(a(K)=0\). We henceforth assume \(\eps>0\). We also assume throughout
that \(\eps_0>0\) is small. For \(|a|\le 1/4\), the translated body \(K-a\)
contains the origin in its interior. Indeed, since \(B_{1-\eps}\subset K\),
choosing \(\eps_0<1/4\) gives
\[
 B_{1-\eps-|a|}\subset K-a\subset B_{1+\eps+|a|}.
\]
Thus, for \(|a|\le1/4\), all these translated bodies are trapped between
two fixed balls centered at the origin. In particular, the radial function
of \(K-a\) with respect to the origin is well-defined. Define
\[
 \Phi_K(a)
 :=\int_{S^2}\bigl(\rho_{K-a}(\omega)-1\bigr)\omega\,d\sigma(\omega).
\]
The map \(a\mapsto\Phi_K(a)\) is continuous, in fact locally Lipschitz. If
\(|a|,|b|\le1/4\), then
$
 d_H(K-a,K-b)=|a-b|,
$
and Lemma~\ref{lem:uniform-radial-stability}, applied with fixed inner and
outer radii, gives
\[
 \|\rho_{K-a}-\rho_{K-b}\|_{L^\infty(S^2)}\le C|a-b|.
\]
This implies the asserted continuity of \(\Phi_K\).

We first compute the map \(\Phi_K\) in the case \(K=B\). The radial function of \(B-a\)
is determined by
$
 |r\omega+a|^2=1.
$
Thus, for \(|a|<1\),
\[
 \rho_{B-a}(\omega)
 =-a\cdot\omega+
 \sqrt{1-|a|^2+(a\cdot\omega)^2}.
\]
Consequently, uniformly for \(\omega\in S^2\),
$\rho_{B-a}(\omega)-1
 =-a\cdot\omega+O(|a|^2).
$
Using
\[
 \int_{S^2}\omega_i\omega_j\,d\sigma
 =\frac{4\pi}{3}\delta_{ij},
\]
we obtain
\begin{equation}\label{eq:Phi-ball-expansion}
 \Phi_B(a)
 =-\int_{S^2}(a\cdot\omega)\omega\,d\sigma+O(|a|^2)
 =-\frac{4\pi}{3}a+O(|a|^2).
\end{equation}
We now compare \(K-a\) with \(B-a\). We have
$
        d_H(K-a,B-a)=d_H(K,B)=\varepsilon.
$
Moreover,
\[
        B_{1-\varepsilon-|a|}\subset K-a\subset B_{1+\varepsilon+|a|},
        \qquad
        B_{1-|a|}\subset B-a\subset B_{1+|a|}.
\]
Thus, if \(\varepsilon_0<1/4\) and \(|a|\le 1/4\), then
$
        B_{1/2}\subset K-a,\ B-a\subset B_{3/2}.
$
 Lemma~\ref{lem:uniform-radial-stability} gives
$
        \|\rho_{K-a}-\rho_{B-a}\|_{L^\infty(S^2)}
        \le C\varepsilon.
$
Therefore
\begin{equation}\label{eq:Phi-comparison}
 |\Phi_K(a)-\Phi_B(a)|\le C\eps.
\end{equation}
Combining \eqref{eq:Phi-ball-expansion} and \eqref{eq:Phi-comparison},
we get
\[
 \Phi_K(a)\cdot a
 \le
 -\frac{4\pi}{3}|a|^2+C|a|^3+C\eps |a|.
\]
We choose \(R>0\) large. After \(R\) has been fixed, we further decrease \(\varepsilon_0\), if necessary,
so that \(R\varepsilon_0\le 1/4\). Thus the map \(\Phi_K\) is well defined
on the whole ball \(\overline{B_{R\varepsilon}}\), and all the preceding estimates
hold uniformly there. If \(|a|=R\eps\), then
\[
 \Phi_K(a)\cdot a
 \le
 -\frac{4\pi}{3}R^2\eps^2+CR^3\eps^3+CR\eps^2.
\]
Taking first \(R\) large enough and then \(\eps_0\) small enough, we obtain
\begin{equation}\label{eq:Phi-inward-boundary}
 \Phi_K(a)\cdot a<0
 \qquad\text{for every } |a|=R\eps.
\end{equation}
We now prove that \(\Phi_K\) has a zero in \(B_{R\varepsilon}\). Assume by contradiction
that
\[
        \Phi_K(a)\neq 0
        \qquad\text{for every }a\in \overline{B_{R\varepsilon}}.
\]
Then the map
\[
        T(a)
        :=
        R\varepsilon\,\frac{\Phi_K(a)}{|\Phi_K(a)|}
\]
is continuous from \(\overline{B_{R\varepsilon}}\) into itself. By Brouwer's fixed point
theorem, there exists \(a_0\in \overline{B_{R\varepsilon}}\) such that \(T(a_0)=a_0\).
Since \(|T(a_0)|=R\varepsilon\), we have \(|a_0|=R\varepsilon\). Moreover,
$
        a_0
        =
        R\varepsilon\,\frac{\Phi_K(a_0)}{|\Phi_K(a_0)|},
$
and hence
$
        \Phi_K(a_0)\cdot a_0
        =
        |\Phi_K(a_0)|\,|a_0|
        >0.
$
This contradicts the boundary inequality
\[
        \Phi_K(a)\cdot a<0
        \qquad\text{for every } |a|=R\varepsilon.
\]
Therefore \(\Phi_K\) has a zero in \(B_{R\varepsilon}\).
Thus there exists \(a(K)\in\R^3\) such that
\[
 |a(K)|\le R\eps,
 \qquad
 \Phi_K(a(K))=0.
\]
Writing \(K^0:=K-a(K)\) and
\(\partial K^0=\{(1+u(\omega))\omega:\omega\in S^2\}\), the identity
\(\Phi_K(a(K))=0\) is precisely
\[
 \int_{S^2}u(\omega)\omega\,d\sigma(\omega)=0.
\]
Finally,
\[
 d_H(K^0,B)
 \le d_H(K,B)+|a(K)|
 \le (1+R)d_H(K,B).
\]
This proves the lemma.
\end{proof}

We shall also use the following elementary consequence of the definition of
the Hausdorff distance for convex bodies. It follows from standard facts on
support functions and the Hausdorff metric; see Schneider~\cite{Schneider2014}.
\begin{lemma}
\label{lem:hausdorff-support-radial}
Let \(K\subset\R^3\) be a convex body with \(0\in\mathrm{int}(K)\), and let
\(h_K\) and \(\rho_K\) denote its support and radial functions,
respectively. If
\[
 d_H(K,B)=\lambda<1,
\]
then
\[
 \|h_K-1\|_{L^\infty(S^2)}=\lambda
\]
and
\[
 B_{1-\lambda}\subset K\subset B_{1+\lambda}.
\]
Consequently,
\[
 1-\lambda\le \rho_K(\omega)\le 1+\lambda
 \qquad\text{for every }\omega\in S^2,
\]
and, writing \(u:=\rho_K-1\),
\[
 \|u\|_{L^\infty(S^2)}= \lambda.
\]
\end{lemma}

The next two lemmas are companion estimates for sequences of convex bodies
converging to the unit ball in the Hausdorff distance. They show that an
almost extremal value of the radial perturbation
$
        u_j:=\rho_{K_j}-1
$
cannot remain isolated at a single point of the sphere: by convexity, it
propagates on a geodesic disk of radius of order \(\lambda_j\).

By
Lemma~\ref{lem:hausdorff-support-radial}, if
$
        d_H(K_j,B)=\lambda_j\to0,
$
then
\[
        \|u_j\|_{L^\infty(S^2)}=\lambda_j,
        \qquad
        B_{1-\lambda_j}\subset K_j\subset B_{1+\lambda_j}.
\]
Thus the hypotheses of Lemmas~\ref{lem:linear-outward}
and~\ref{lem:linear-inward} simply correspond to choosing points where the
radial perturbation is asymptotically extremal, either in the outward or in
the inward direction. We include the proofs for the reader's convenience.

We first record a simple convexity consequence. It is deliberately
formulated with a free height fraction \(\alpha\in(0,1)\), since this is
what improves the final constant.

\begin{lemma}
\label{lem:linear-outward}
Let \(K_j\subset\R^3\) be convex bodies such that
$
 d_H(K_j,B)=\lambda_j\to0.
$
Assume that, for some \(p_j\in S^2\),
$
 \rho_{K_j}(p_j)=1+\lambda_j+o(\lambda_j),
$
where \(\rho_{K_j}\) is the radial function of \(K_j\). Then for every
\(\alpha\in(0,1)\) there exist \(c_\alpha>0\) and \(j_\alpha\) such that,
for \(j\ge j_\alpha\),
\[
 \rho_{K_j}(\omega)\ge 1+\alpha\lambda_j+o(\lambda_j)
\]
uniformly for \(\omega\in D_{c_\alpha\lambda_j}(p_j)\).
\end{lemma}

\begin{proof}
Since \(d_H(K_j,B)=\lambda_j\), Lemma~\ref{lem:hausdorff-support-radial}
gives the exact inclusions
$
 B_{1-\lambda_j}\subset K_j\subset B_{1+\lambda_j}.
$
Fix a unit vector \(q_j\perp p_j\).
Then
$
 Q_j:=(1-\lambda_j)q_j\in K_j,
$
and by assumption
$
 P_j:=(1+\lambda_j+o(\lambda_j))p_j\in K_j.
$
By convexity, \(K_j\) contains the segment \([P_j,Q_j]\).

We now work in the two-dimensional plane spanned by \(p_j\) and \(q_j\), with coordinates
\((x,z)\), where \(z\) is the coordinate in the \(p_j\) direction and \(x\) is the
coordinate in the \(q_j\) direction. We set
\[
        A_j:=1+\lambda_j+o(\lambda_j),
        \qquad
        B_j:=1-\lambda_j.
\]
Then the endpoints of the segment are
$
        P_j=(0,A_j),
        Q_j=(B_j,0).
$
Let \(\theta\) be the angle from \(p_j\), and let
$
        \omega(\theta)=\sin\theta\,q_j+\cos\theta\,p_j.
$
The ray \(\{\rho\omega(\theta):\rho\ge0\}\) intersects the line through \(P_j\)
and \(Q_j\) at the radial coordinate
\[
        \rho_j(\theta)
        =
        \frac{A_jB_j}{B_j\cos\theta+A_j\sin\theta}.
\]
The point \(\rho_j(\theta)\omega(\theta)\) belongs to the segment
\([P_j,Q_j]\), hence to \(K_j\). Therefore
$
        \rho_{K_j}(\omega(\theta))\ge \rho_j(\theta).
$
For \(\theta\le c_\alpha\lambda_j\),
\[
        \cos\theta=1+O(\lambda_j^2),
        \qquad
        \sin\theta\le c_\alpha\lambda_j+O(\lambda_j^3).
\]
Hence
\[
\begin{aligned}
        B_j\cos\theta+A_j\sin\theta
        &\le
        (1-\lambda_j)(1+O(\lambda_j^2))
        +
        (1+\lambda_j+o(\lambda_j))(c_\alpha\lambda_j+O(\lambda_j^3)) \\
        &=
        1-(1-c_\alpha)\lambda_j+o(\lambda_j).
\end{aligned}
\]
Moreover,
$
        A_jB_j
        =
        (1+\lambda_j+o(\lambda_j))(1-\lambda_j)
        =
        1+o(\lambda_j).
$
Therefore
\[
        \rho_j(\theta)
        \ge
        \frac{1+o(\lambda_j)}
        {1-(1-c_\alpha)\lambda_j+o(\lambda_j)}
        =
        1+(1-c_\alpha)\lambda_j+o(\lambda_j).
\]
We choose \(c_\alpha>0\) so small that \(1-c_\alpha>\alpha\). Since
\(\rho_{K_j}(\omega(\theta))\ge \rho_j(\theta)\), we obtain
\[
        \rho_{K_j}(\omega(\theta))
        \ge
        1+\alpha\lambda_j+o(\lambda_j)\qquad \text{uniformly for}\ (0\le\theta\le c_\alpha\lambda_j).
\]
The preceding argument was written in one meridian plane. Since \(q_j\)
was arbitrary in \(p_j^\perp\), the same estimate holds in every
meridian plane through \(p_j\), hence on the whole geodesic disk
\(D_{c_\alpha\lambda_j}(p_j)\).
\end{proof}

\begin{lemma}
\label{lem:linear-inward}
Let \(K_j\subset\R^3\) be convex bodies such that
$
 d_H(K_j,B)= \lambda_j\to 0.
$
 Assume that,
for some \(p_j\in S^2\),
$
 \rho_{K_j}(p_j)=1-\lambda_j+o(\lambda_j).
$
Then, for every \(\alpha\in(0,1)\), there exist \(c_\alpha>0\) and
\(j_\alpha\) such that, for \(j\ge j_\alpha\),
\[
 \rho_{K_j}(\omega)\le 1-\alpha\lambda_j+o(\lambda_j)
\]
uniformly for \(\omega\in D_{c_\alpha\lambda_j}(p_j)\).
\end{lemma}

\begin{proof}
Since \(B_{1-\lambda_j}\subset K_j\), the radial function satisfies, for all $j$,
$
        \rho_{K_j}(p_j)\ge 1-\lambda_j
$,
and we can write
\[
        s_j:=\rho_{K_j}(p_j)=1-\lambda_j+\eta_j,
        \qquad
        0\le \eta_j=o(\lambda_j).
\]
Let
\[
        x_j:=s_jp_j\in\partial K_j,
\]
and choose, by convexity of \(K_j\), an outer unit normal
\(\nu_j\in S^2\) to a supporting hyperplane of \(K_j\) at \(x_j\). Then
\[
        K_j
        \subset
        \{x\in\mathbb R^3:x\cdot\nu_j\le x_j\cdot\nu_j\}
        =
        \{x\in\mathbb R^3:x\cdot\nu_j\le s_j\,p_j\cdot\nu_j\}.
\]
Applying this inclusion to the point
$
        z=(1-\lambda_j)\nu_j\in B_{1-\lambda_j}\subset K_j
$
gives
\[
        1-\lambda_j\le s_j\,p_j\cdot\nu_j.
\]
Writing \(p_j\cdot\nu_j=\cos\gamma_j\) with \(\gamma_j\in[0,\pi]\), we get
\[
        \cos\gamma_j
        \ge
        \frac{1-\lambda_j}{s_j}
        =
        \frac{1-\lambda_j}{1-\lambda_j+\eta_j}
        =
        1-\frac{\eta_j}{1-\lambda_j+\eta_j}.
\]
Since \(\cos\gamma_j\le 1\), it follows that
\[
        0\le 1-\cos\gamma_j
        \le
        \frac{\eta_j}{1-\lambda_j+\eta_j}
        =
        o(\lambda_j).
\]
In particular, \(\gamma_j\to0\). This implies
\begin{equation}\label{eq:gamma-small-sqrt-lambda}
        \gamma_j=o\!\left(\sqrt{\lambda_j}\right).
\end{equation}
We fix \(c_\alpha>0\). We shall prove the estimate on
\(D_{c_\alpha\lambda_j}(p_j)\), after possibly increasing \(j_\alpha\).
Let now \(\omega\in S^2\) with
\begin{equation}\label{eq:theta-geodesic-smallness}
        \theta_j:=d_{S^2}(\omega,p_j)\le c_\alpha\lambda_j.
\end{equation}
We denote by \(\beta_j\in[0,\pi]\) the geodesic angle between
\(\omega\) and \(\nu_j\), so that
\[
        \omega\cdot\nu_j=\cos\beta_j.
\]
The function \(\nu\mapsto d_{S^2}(\nu,\nu_j)\) is \(1\)-Lipschitz on
\(S^2\), hence
\[
        |\beta_j-\gamma_j|
        =
        \bigl|d_{S^2}(\omega,\nu_j)-d_{S^2}(p_j,\nu_j)\bigr|
        \le
        d_{S^2}(\omega,p_j)
        =
        \theta_j.
\]
Combining \eqref{eq:gamma-small-sqrt-lambda} and \eqref{eq:theta-geodesic-smallness}
we get
$
        \beta_j=o(\sqrt{\lambda_j})
$
uniformly for \(\omega\in D_{c_\alpha\lambda_j}(p_j)\). In particular,
\(\cos\beta_j>0\) for \(j\) large and \(\omega\) in the disk.

Since the supporting plane to \(K_j\) at the point \(s_jp_j\) contains
\(K_j\) in the half-space determined by its outer normal, every point
\(x\in K_j\) satisfies
$
        x\cdot \nu_j \le s_jp_j\cdot \nu_j.
$
Applying this to
$
        x=\rho_{K_j}(\omega)\omega\in K_j,
$
we obtain
$
        \rho_{K_j}(\omega)\,\omega\cdot\nu_j
        \le
        s_j\,p_j\cdot\nu_j.
$
By the definitions of the angles \(\beta_j\) and \(\gamma_j\), this is
equivalent to
$
        \rho_{K_j}(\omega)\cos\beta_j
        \le
        s_j\cos\gamma_j.
$
Therefore
\[
        \rho_{K_j}(\omega)
        \le
        s_j\,\frac{\cos\gamma_j}{\cos\beta_j}.
\]
By the mean value theorem,
$
        |\cos\beta_j-\cos\gamma_j|
        \le
        \sin\xi_j\,|\beta_j-\gamma_j|
$
for some \(\xi_j\) between \(\beta_j\) and \(\gamma_j\). Since
$
        |\beta_j-\gamma_j|\le \theta_j,
$
we have
$
        \beta_j\le \gamma_j+\theta_j,
$
and hence, every \(\xi_j\) between \(\beta_j\) and \(\gamma_j\) satisfies
$
        \xi_j\le \gamma_j+\theta_j.
$
For \(j\) large all these angles are small, and therefore
$
        \sin\xi_j\le \gamma_j+\theta_j.
$
Thus
\[
        |\cos\beta_j-\cos\gamma_j|
        \le
        (\gamma_j+\theta_j)\theta_j
        =
        \theta_j\gamma_j+\theta_j^2.
\]
All the following estimates are uniform for
\(\omega\in D_{c_\alpha\lambda_j}(p_j)\). Since
\[
        \theta_j\le c_\alpha\lambda_j,
        \qquad
        \gamma_j=o(\sqrt{\lambda_j}),
\]
we get
$
        \theta_j\gamma_j+\theta_j^2=o(\lambda_j).
$
Together with
\[
        1-\cos\gamma_j=o(\lambda_j),
        \qquad
        \cos\beta_j=\cos\gamma_j+o(\lambda_j),
\]
this gives
$
        \frac{\cos\gamma_j}{\cos\beta_j}
        =
        1+o(\lambda_j)
$
uniformly for \(\omega\in D_{c_\alpha\lambda_j}(p_j)\).
Therefore
\[
\begin{aligned}
        \rho_{K_j}(\omega)
        &\le
        s_j\bigl(1+o(\lambda_j)\bigr)  \\
        &=
        \bigl(1-\lambda_j+\eta_j\bigr)\bigl(1+o(\lambda_j)\bigr) \\
        &=
        1-\lambda_j+o(\lambda_j),
\end{aligned}
\]
where in the last equality we used \(\eta_j=o(\lambda_j)\). All
\(o(\lambda_j)\)-terms are uniform for
\(\omega\in D_{c_\alpha\lambda_j}(p_j)\). Hence, for every fixed
\(\alpha\in(0,1)\),
\[
        \rho_{K_j}(\omega)
        \le
        1-\alpha\lambda_j+o(\lambda_j)
\]
uniformly for \(\omega\in D_{c_\alpha\lambda_j}(p_j)\), for \(j\) large.
\end{proof}
We now record a direct consequence of the two previous lemmas.
It says that, for sequences of convex bodies converging to the unit ball in
the Hausdorff distance, an almost extremal radial displacement propagates on
a small spherical disk of radius comparable to the Hausdorff distance. This
will be used later in the proof of Theorem~\ref{thm:rough-local-gap}.
\begin{corollary}
\label{cor:one-propagated-disk}
Let \((K_j)_j\) be convex bodies such that \(K_j\to B\) in Hausdorff
distance, and write
\[
        \partial K_j=\{(1+u_j(\omega))\omega:\omega\in S^2\},
        \qquad
        \lambda_j=d_H(K_j,B).
\]
For every \(\alpha\in(0,1)\), there exist points \(p_j\in S^2\), a sign
\(\sigma_j\in\{-1,1\}\), a constant \(c_\alpha>0\) independent of \(j\),
and a sequence \(\eta_j\to 0\), such that
\[
        \sigma_j u_j(\omega)
        \ge
        \alpha\lambda_j(1-\eta_j)
        \qquad\text{for every }\omega\in D_{c_\alpha\lambda_j}(p_j),
\]
for every \(j\) large enough.
\end{corollary}

\begin{remark}
The lower bound in Corollary~\ref{cor:one-propagated-disk} is pointwise on
the disk. In particular, it holds \(d\sigma\)-a.e.\ on the disk, and
therefore it satisfies the hypothesis of
Lemma~\ref{lem:one-disk-logarithmic-estimate} without any further
regularization.
\end{remark}

We shall use the following logarithmic capacitary estimate on \(S^2\).
It gives the energy cost of imposing a height \(h\) on a spherical disk of
geodesic radius \(a\), when the \(L^2\)-norm of the function is small
compared with \(h\). Since \(S^2\) is two-dimensional, the relevant capacity
is logarithmic, of order \(1/|\log a|\); hence the natural energy scale is
$\displaystyle
        \frac{h^2}{|\log a|}.
$
This is the intrinsic two-dimensional analogue of the logarithmic estimate
of Hall, Hayman and Weitsman~\cite[eq.~(4.6)]{HHW1991}, not of their
three-dimensional Newtonian estimate \cite[Thm.~3.1]{HHW1991}.

\begin{lemma}
\label{lem:one-disk-logarithmic-estimate}
Let \(p_j\in \mathbb S^2\) and let \(a_j>0\) be such that
$
        a_j\to0.
$
We denote by
\[
        D_{a_j}(p_j)
        :=
        \{\omega\in\mathbb S^2:\ d_{\mathbb S^2}(\omega,p_j)<a_j\}
\]
the geodesic disk of center \(p_j\) and radius \(a_j\) on \(\mathbb S^2\).
Let \(h_j>0\), and let \(v_j\in H^1(\mathbb S^2)\) satisfy
\[
        v_j \ge h_j
        \qquad\text{for }d\sigma\text{-a.e.\ on }D_{a_j}(p_j).
\]
Assume moreover that
\[
        \|v_j\|_{L^2(\mathbb S^2)} = o(h_j).
\]
Then
\[
        \int_{\mathbb S^2} |\nabla_{\mathbb S^2} v_j|^2\,d\sigma
        \ge
        \frac{2\pi h_j^2}{|\log a_j|}\,(1+o(1)).
\]

In particular, let \(\lambda_j>0\) be such that \(\lambda_j\to0\). If
\[
        h_j=\alpha\lambda_j(1+o(1)),
        \qquad
        a_j=c\lambda_j(1+o(1)),
\]
with \(\alpha,c>0\), then
\[
        \int_{\mathbb S^2} |\nabla_{\mathbb S^2} v_j|^2\,d\sigma
        \ge
        2\pi\alpha^2
        \frac{\lambda_j^2}{|\log\lambda_j|}
        (1+o(1)).
\]
\end{lemma}

\begin{proof}
We use geodesic polar coordinates \((\theta,\varphi)\) around \(p_j\). Thus
\[
        d\sigma=\sin\theta\,d\theta\,d\varphi,
        \qquad
        0<\theta<\pi,\quad 0<\varphi<2\pi,
\]
and the disk \(D_{a_j}(p_j)\) is described by \(0<\theta<a_j\).

Since \(v_j\ge h_j\) for \(d\sigma\)-a.e.\ \(\omega\in D_{a_j}(p_j)\), and since
\(\sin\theta>0\) for \(0<\theta<a_j\), Fubini's theorem implies that, for a.e.\ \(\theta\in(0,a_j)\),
\[
        v_j(\theta,\varphi)\ge h_j
        \qquad\text{for a.e.\ }\varphi\in(0,2\pi).
\]
In particular, we can choose
$
        b_j\in(a_j/2,a_j)
$
such that
\[
        v_j(b_j,\varphi)\ge h_j
        \qquad\text{for a.e.\ }\varphi\in(0,2\pi).
\]
We define the circular mean
\[
        m_j(\theta)
        :=
        \frac1{2\pi}\int_0^{2\pi} v_j(\theta,\varphi)\,d\varphi.
\]
Then
$
        m_j(b_j)\ge h_j.
$
We fix \(0<R_0<R_1<\pi/4\). Since \(a_j\to 0\), for \(j\) large, we have
\(a_j<R_0\). By Jensen's inequality,
\[|m_j(\theta)|^2
\le
\frac1{2\pi}\int_0^{2\pi}|v_j(\theta,\varphi)|^2\,d\varphi\]
and so
\[
\begin{aligned}
        \int_{R_0}^{R_1} |m_j(\theta)|^2\sin\theta\,d\theta
        &\le
        \frac1{2\pi}
        \int_{R_0}^{R_1}\int_0^{2\pi}
        |v_j(\theta,\varphi)|^2
        \sin\theta\,d\varphi\,d\theta       \\
        &\le
        C\|v_j\|_{L^2(\mathbb S^2)}^2
        =
        o(h_j^2).
\end{aligned}
\]
Therefore there exists \(\rho_j\in(R_0,R_1)\) such that
$
        |m_j(\rho_j)|=o(h_j).
$
Since \(v_j\in H^1(\mathbb S^2)\), the circular mean \(m_j\) belongs to
\(H^1_{\mathrm{loc}}((0,\pi))\), and for a.e.\ \(\theta\)
\[
        m_j'(\theta)
        =
        \frac1{2\pi}
        \int_0^{2\pi}\partial_\theta v_j(\theta,\varphi)\,d\varphi.
\]
By Jensen's inequality again,
\[
        |m_j'(\theta)|^2
        \le
        \frac1{2\pi}
        \int_0^{2\pi}
        |\partial_\theta v_j(\theta,\varphi)|^2\,d\varphi,
\]
for a.e.\ \(\theta\in(0,\pi)\). Hence
\[
\begin{aligned}
        \int_{D_{R_1}(p_j)\setminus D_{b_j}(p_j)}
        |\nabla_{\mathbb S^2} v_j|^2\,d\sigma
        &\ge
        \int_{b_j}^{R_1}\int_0^{2\pi}
        |\partial_\theta v_j(\theta,\varphi)|^2
        \sin\theta\,d\varphi\,d\theta                                      \\
        &\ge
        2\pi
        \int_{b_j}^{R_1}
        |m_j'(\theta)|^2\sin\theta\,d\theta                                  \\
        &\ge
        2\pi
        \int_{b_j}^{\rho_j}
        |m_j'(\theta)|^2\sin\theta\,d\theta.
\end{aligned}
\]
Writing
\[
        m_j(b_j)-m_j(\rho_j)
        =
        -\int_{b_j}^{\rho_j} m_j'(\theta)\,d\theta ,
\]
and applying the Cauchy--Schwarz inequality,
\[
\begin{aligned}
        \left(m_j(b_j)-m_j(\rho_j)\right)^2
        &=
        \left(
        \int_{b_j}^{\rho_j}
        -m_j'(\theta)\,\sqrt{\sin\theta}\,
        \frac{d\theta}{\sqrt{\sin\theta}}
        \right)^2  \\
        &\le
        \left(
        \int_{b_j}^{\rho_j}
        |m_j'(\theta)|^2\sin\theta\,d\theta
        \right)
        \left(
        \int_{b_j}^{\rho_j}
        \frac{d\theta}{\sin\theta}
        \right).
\end{aligned}
\]
Therefore
\[
        \int_{b_j}^{\rho_j}
        |m_j'(\theta)|^2\sin\theta\,d\theta
        \ge
        \frac{
        \left(m_j(b_j)-m_j(\rho_j)\right)^2
        }{
        \displaystyle\int_{b_j}^{\rho_j}\frac{d\theta}{\sin\theta}
        }.
\]
Since \(m_j(b_j)\ge h_j\) and \(|m_j(\rho_j)|=o(h_j)\), we have
\[
        m_j(b_j)-m_j(\rho_j)
        \ge
        h_j-|m_j(\rho_j)|
        =
        h_j(1-o(1)).
\]
Hence
$        \left(m_j(b_j)-m_j(\rho_j)\right)^2
        \ge
        h_j^2(1-o(1)).
$
Moreover, since \(b_j\in(a_j/2,a_j)\) and \(\rho_j\in(R_0,R_1)\),
\[
        \int_{b_j}^{\rho_j}\frac{d\theta}{\sin\theta}
        =
        |\log a_j|+O(1).
\]
Therefore
\[
\begin{aligned}
        \int_{\mathbb S^2}|\nabla_{\mathbb S^2}v_j|^2\,d\sigma
        &\ge
        \int_{D_{R_1}(p_j)\setminus D_{b_j}(p_j)}
        |\nabla_{\mathbb S^2}v_j|^2\,d\sigma                                \\
                &\ge
        2\pi
        \frac{h_j^2(1-o(1))}{|\log a_j|+O(1)}                                \\
        &=
        \frac{2\pi h_j^2}{|\log a_j|}(1+o(1)).
\end{aligned}
\]
This proves the first assertion.

If now
\[
        h_j=\alpha\lambda_j(1+o(1)),
        \qquad
        a_j=c\lambda_j(1+o(1)),
\]
with \(\alpha,c>0\), then
$
        h_j^2=\alpha^2\lambda_j^2(1+o(1)),
        |\log a_j|=|\log\lambda_j|+O(1).
$
Replacing into the previous estimate yields
\[
        \int_{\mathbb S^2} |\nabla_{\mathbb S^2} v_j|^2\,d\sigma
        \ge
        2\pi\alpha^2
        \frac{\lambda_j^2}{|\log\lambda_j|}
        (1+o(1)).
\]
\end{proof}

\begin{remark}
In the application below the hypothesis of
Lemma~\ref{lem:one-disk-logarithmic-estimate} is automatically satisfied.
Indeed, by Corollary~\ref{cor:one-propagated-disk}, the lower bound on
\(v_j=\sigma_j u_j\) holds pointwise on the disk
\(D_{c_\alpha\lambda_j}(p_j)\). In particular, it holds \(d\sigma\)-a.e.\ on
that disk.
\end{remark}

The next lemma uses the volume constraint together with the exact
elimination of the first spherical harmonics provided by the recentering
condition \eqref{eq:first-moment-zero}.

\begin{lemma}
\label{lem:spectral}
Let \(u_j\to0\) in \(W^{1,2}(S^2)\cap L^\infty(S^2)\). Assume that
\begin{equation}\label{eq:volume-constraint-uj}
        \frac13\int_{S^2}(1+u_j)^3\,d\sigma=\frac{4\pi}{3}
\end{equation}
and
\[
        \int_{S^2}u_j(\omega)\,\omega\,d\sigma(\omega)=0.
\]
Then
\[
        \int_{S^2}
        \bigl(|\nabla_{S^2}u_j|^2-2u_j^2\bigr)\,d\sigma
        \ge
        4\int_{S^2}u_j^2\,d\sigma
        +o(\|u_j\|_{L^2(S^2)}^2).
\]
\end{lemma}

\begin{proof}
Set
\[
        \bar u_j:=\frac1{4\pi}\int_{S^2} u_j\,d\sigma,
        \qquad
        w_j:=u_j-\bar u_j.
\]
Expanding the volume constraint by \eqref{eq:volume-constraint-uj} gives
\[
        \int_{S^2}u_j\,d\sigma
        =
        -\int_{S^2}u_j^2\,d\sigma
        -
        \frac13\int_{S^2}u_j^3\,d\sigma.
\]
Since \(\|u_j\|_{L^\infty(S^2)}\to0\), we have
\[
        \left|\int_{S^2}u_j^3\,d\sigma\right|
        \le
        \|u_j\|_{L^\infty(S^2)}
        \int_{S^2}u_j^2\,d\sigma
        =
        o(\|u_j\|_{L^2(S^2)}^2).
\]
Therefore
\[
        \int_{S^2}u_j\,d\sigma
        =
        -\int_{S^2}u_j^2\,d\sigma
        +
        o(\|u_j\|_{L^2(S^2)}^2),
\]
and hence
$
        |\bar u_j|
        =
        O(\|u_j\|_{L^2(S^2)}^2).
$
Consequently,
\[
        \int_{S^2}w_j^2\,d\sigma
        =
        \int_{S^2}u_j^2\,d\sigma
        -
        4\pi\bar u_j^2
        =
        \int_{S^2}u_j^2\,d\sigma
        +
        o(\|u_j\|_{L^2(S^2)}^2),
\]
while
\[
        \int_{S^2}|\nabla_{S^2}w_j|^2\,d\sigma
        =
        \int_{S^2}|\nabla_{S^2}u_j|^2\,d\sigma.
\]
Moreover \(w_j\) has zero average and
\[
        \int_{S^2}w_j(\omega)\,\omega\,d\sigma(\omega)
        =
        \int_{S^2}u_j(\omega)\,\omega\,d\sigma(\omega)
        -
        \bar u_j\int_{S^2}\omega\,d\sigma
        =
        0.
\]
Thus the spherical-harmonic expansion of \(w_j\) has no modes of degree
\(0\) and no modes of degree \(1\). Indeed, the spectrum of
\(-\Delta_{S^2}\) is \(\ell(\ell+1)\), \(\ell=0,1,2,\dots\), and the
eigenspace corresponding to the eigenvalue \(2\) is spanned by the coordinate
functions \(\omega_1,\omega_2,\omega_3\); see, for instance,
Chavel~\cite[Ch.~I]{Chavel1984}. After excluding the modes of degrees~\(0\) and~\(1\), the first non-zero eigenvalue of
\(-\Delta_{S^2}\) is the degree-two eigenvalue \(6\), and
\[
        \int_{S^2}|\nabla_{S^2}w_j|^2\,d\sigma
        \ge
        6\int_{S^2}w_j^2\,d\sigma.
\]
Consequently,
\[
\begin{aligned}
        \int_{S^2}\bigl(|\nabla_{S^2}u_j|^2-2u_j^2\bigr)\,d\sigma
        &=
        \int_{S^2}|\nabla_{S^2}w_j|^2\,d\sigma
        -
        2\int_{S^2}u_j^2\,d\sigma                                      \\
        &\ge
        6\int_{S^2}w_j^2\,d\sigma
        -
        2\int_{S^2}u_j^2\,d\sigma                                      \\
        &=
        4\int_{S^2}u_j^2\,d\sigma
        +
        o(\|u_j\|_{L^2(S^2)}^2).
\end{aligned}
\]
\end{proof}

\begin{remark}
Without the first-moment condition, the degree-one spherical harmonics are
present. They correspond infinitesimally to translations of the ball. After
removing only the mean, the Poincaré inequality starts at the eigenvalue
\(2\), and the quadratic form
\[
 \int_{S^2}\bigl(|\nabla u|^2-2u^2\bigr)\,d\sigma
\]
would no longer control \(\|u\|_2^2\). The recentering condition removes
these translation modes and restores the degree-two spectral gap.
\end{remark}

We record here the precise form of the quasi-spherical perimeter
expansion needed below. This is the elementary second-order expansion
underlying Fuglede's stability theorem for convex or nearly spherical
domains (see~\cite{Fuglede1989}). If \(u\in W^{1,\infty}(S^2)\),
\(1+u>0\), and
\[
 \partial F=\{(1+u(\omega))\omega:\omega\in S^2\},
\]
then the area formula for Lipschitz parametrizations
(see, e.g., Maggi~\cite[Thm.~13.1]{Maggi2012}) yields
\[
 P(F)
 =
 \int_{S^2}
 (1+u)^2
 \sqrt{1+\frac{|\nabla_{S^2}u|^2}{(1+u)^2}}
 \,d\sigma.
\]
A Taylor expansion of the integrand yields the desired second-order
formula when both \(u\) and \(\nabla_{S^2}u\) are uniformly small.

\begin{lemma}
\label{lem:fuglede-small-slope}
Let \(F_j\subset\mathbb R^3\) be a sequence of sets of volume
\(4\pi/3\), written as radial Lipschitz graphs
\[
 \partial F_j=\{(1+u_j(\omega))\omega:\omega\in S^2\},
 \qquad u_j\in W^{1,\infty}(S^2),
\]
with \(1+u_j>0\). Assume that
\[
 \eta_j:=\|u_j\|_{L^\infty(S^2)}+
        \|\nabla_{S^2}u_j\|_{L^\infty(S^2)}\to0.
\]
Then
\[
 P(F_j)-4\pi
 =
 \frac12\int_{S^2}\bigl(|\nabla_{S^2}u_j|^2-2u_j^2\bigr)\,d\sigma
 +R_j,
\]
where
\[
 |R_j|\le C\eta_j\int_{S^2}\bigl(u_j^2+|\nabla_{S^2}u_j|^2\bigr)\,d\sigma.
\]
\end{lemma}

\begin{proof}
For \(|u|+|\xi|\le\eta_j\) and \(j\) large, define
\[
 f(u,\xi):=(1+u)^2\sqrt{1+\frac{|\xi|^2}{(1+u)^2}}.
\]
Since \(f\) is smooth in a fixed neighborhood of \((0,0)\) and
\(f(0,0)=1\), \(\partial_u f(0,0)=2\), \(\partial_\xi f(0,0)=0\),
\(\partial^2_{uu}f(0,0)=2\), \(\partial^2_{\xi\xi}f(0,0)=I\),
Taylor's formula with integral remainder gives
\[
 f(u,\xi)=1+2u+u^2+\frac12|\xi|^2+r(u,\xi),
\]
where
\[
 |r(u,\xi)|\le C(|u|+|\xi|)\bigl(u^2+|\xi|^2\bigr)
 \le C\eta_j\bigl(u^2+|\xi|^2\bigr)
\]
uniformly for \(|u|+|\xi|\le\eta_j\).
Applying the area formula and integrating term by term, one gets
\begin{equation}\label{eq:perimeter-expansion-with-R1}
\begin{aligned}
 P(F_j)
 &=
 \int_{S^2}f(u_j,\nabla_{S^2}u_j)\,d\sigma  \\
 &=
 4\pi+2\int_{S^2}u_j\,d\sigma
 +\int_{S^2}u_j^2\,d\sigma
 +\frac12\int_{S^2}|\nabla_{S^2}u_j|^2\,d\sigma
 +R_j^{(1)},
\end{aligned}
\end{equation}
where
\begin{equation}\label{eq:R1-bound}
 R_j^{(1)}:=\int_{S^2}r(u_j,\nabla_{S^2}u_j)\,d\sigma,
 \qquad
 |R_j^{(1)}|
 \le
 C\eta_j\int_{S^2}\bigl(u_j^2+|\nabla_{S^2}u_j|^2\bigr)\,d\sigma.
\end{equation}
Using the volume constraint \eqref{eq:volume-constraint-uj}, we get
\[
        3\int_{S^2}u_j\,d\sigma
        +3\int_{S^2}u_j^2\,d\sigma
        +\int_{S^2}u_j^3\,d\sigma
        =0,
\]
and therefore
\begin{equation}\label{eq:volume-linear-term}
 \int_{S^2}u_j\,d\sigma
 =
 -\int_{S^2}u_j^2\,d\sigma
 -\frac13\int_{S^2}u_j^3\,d\sigma.
\end{equation}
It follows from \eqref{eq:volume-linear-term} that
\begin{equation}\label{eq:linear-terms-volume-reduction}
 2\int_{S^2}u_j\,d\sigma
 +\int_{S^2}u_j^2\,d\sigma
 =
 -\int_{S^2}u_j^2\,d\sigma
 +R_j^{(2)},
\end{equation}
where
\begin{equation}\label{eq:R2-definition-bound}
 R_j^{(2)}
 :=
 -\frac23\int_{S^2}u_j^3\,d\sigma,
 \qquad
 |R_j^{(2)}|
 \le
 \frac23\eta_j\int_{S^2}u_j^2\,d\sigma.
\end{equation}
Combining \eqref{eq:perimeter-expansion-with-R1} with
\eqref{eq:linear-terms-volume-reduction}, we obtain
\[
 P(F_j)-4\pi
 =
 \frac12\int_{S^2}
 \bigl(|\nabla_{S^2}u_j|^2-2u_j^2\bigr)\,d\sigma
 +R_j,
\]
where $R_j:=R_j^{(1)}+R_j^{(2)}$
satisfies
$\displaystyle
 |R_j|
 \le
 C\eta_j\int_{S^2}
 \bigl(u_j^2+|\nabla_{S^2}u_j|^2\bigr)\,d\sigma.
$
\end{proof}

\begin{remark}
The estimate on \(R_j\) is a relative second-order estimate in the
small-slope regime. In the convex class, after the volume constraint has controlled the average
of \(u\) and the chosen normalization has removed the first spherical
harmonics, the quadratic form
\[
 u\mapsto \int_{S^2}(|\nabla_{S^2}u|^2-2u^2)\,d\sigma
\]
is coercive up to an error of order \(o(\|u\|_{L^2}^2)\), by
Lemma~\ref{lem:spectral}. This will be used in the proof of Theorem~\ref{thm:rough-local-gap}.
\end{remark}

We are going to show that in the convex class, the
small-slope condition required by Lemma~\ref{lem:fuglede-small-slope}
is automatic for sequences converging to the unit ball in Hausdorff
distance. The key estimate is essentially
contained in Fuglede~\cite[Lemma~2.2]{Fuglede1989}; we only translate
his notation to the present Hausdorff normalization.

\begin{lemma}
\label{lem:automatic-slope}
Let $K\subset\R^3$ be a convex body with $0\in\mathrm{int}(K)$ and
\(d_H(K,B)=\lambda\in(0,1/2)\). Let \(\rho_K\) be the radial function
of \(K\), and write
\[
 \partial K=\{\rho_K(\omega)\omega:\omega\in S^2\},
 \qquad u:=\rho_K-1.
\]
Then \(u\in W^{1,\infty}(S^2)\),
\[
        \|u\|_{L^\infty(S^2)}=\lambda,
\]
and
\[
 \|\nabla_{S^2}u\|_{L^\infty(S^2)}
 \;\le\;
 \frac{2\sqrt{\lambda}\,(1+\lambda)}{1-\lambda}
 \;=\;
 2\sqrt{\lambda}\,\bigl(1+O(\lambda)\bigr).
\]
In particular,
\[
 \|u\|_{L^\infty(S^2)}+\|\nabla_{S^2}u\|_{L^\infty(S^2)}
 \;=\;
 O(\sqrt{\lambda}).
\]
\end{lemma}

\begin{proof}
By Lemma~\ref{lem:hausdorff-support-radial},
\[
        B_{1-\lambda}\subset K\subset B_{1+\lambda},
        \qquad
        \|u\|_{L^\infty(S^2)}=\lambda.
\]
In the notation of Fuglede~\cite[Definition~2.1]{Fuglede1989}, these
inclusions mean that the spherical deviation of \(K\), with respect to the
present center, is \(d=\lambda\). Fuglede's Lemma~2.2 is stated for normalized
convex bodies, but its proof uses only the convexity of the body and the
two inclusions above. Therefore the same estimate applies here to the radial
parametrization
\[
        \partial K=\{(1+u(\omega))\omega:\omega\in S^2\}.
\]
It gives
\[
        \|\nabla_{S^2}u\|_{L^\infty(S^2)}
        \le
        2\sqrt{\lambda}\,\frac{1+\lambda}{1-\lambda}.
\]
In particular \(u\in W^{1,\infty}(S^2)\). Since \(\lambda<1/2\),
\[
        \frac{1+\lambda}{1-\lambda}=1+O(\lambda),
\]
and hence
\[
        \|\nabla_{S^2}u\|_{L^\infty(S^2)}
        \le
        2\sqrt{\lambda}(1+O(\lambda)).
\]
Together with \(\|u\|_{L^\infty(S^2)}=\lambda\), this yields
\[
        \|u\|_{L^\infty(S^2)}
        +
        \|\nabla_{S^2}u\|_{L^\infty(S^2)}
        =
        O(\sqrt{\lambda}).
\]
\end{proof}

\begin{remark}
\label{rem:cone-sharpness}
The exponent $1/2$ in Lemma~\ref{lem:automatic-slope} is sharp.
Consider the convex prolate cone
\(K_\lambda:=\mathrm{conv}(B\cup\{\pm(1+\lambda)e_3\})\).
A direct computation in geodesic polar coordinates around $e_3$ gives,
for $\theta\in[0,\theta_0]$ with $\cos\theta_0=1/(1+\lambda)$,
\[
 1+u(\theta)
 =
 \frac{1+\lambda}{\cos\theta+\sin\theta\,\sqrt{2\lambda+\lambda^2}},
\]
and $u(\theta)=0$ for $\theta\in[\theta_0,\pi-\theta_0]$.
Differentiating at $\theta=0^+$, we get
\[
 \|\nabla_{S^2}u\|_{L^\infty(S^2)}
 =
 \sqrt{2\lambda}\,\bigl(1+O(\lambda)\bigr),
\]
which matches the upper bound of Lemma~\ref{lem:automatic-slope} up to
the constant $\sqrt 2$.
\end{remark}

We recall the admissible class used throughout the paper. We set
\[
        \cA
        :=
        \left\{
        K\subset\R^3:\ K \text{ is a compact convex body and }
        |K|=\frac{4\pi}{3}
        \right\}.
\]
In particular, when we write \(F_j\to B\) in the
Hausdorff distance below, we are considering representatives which converge
to the centered unit ball \(B=B_1(0)\).

\begin{corollary}
\label{cor:fuglede-automatic}
Let $(F_j)_j\subset\cA$ be a sequence with $F_j\to B$ in
Hausdorff distance, and write
\(\partial F_j=\{(1+u_j(\omega))\omega:\omega\in S^2\}\),
\(\lambda_j:=d_H(F_j,B)\). Then
\[
 \|u_j\|_{L^\infty(S^2)}+\|\nabla_{S^2}u_j\|_{L^\infty(S^2)}
 \;\le\;
 C\sqrt{\lambda_j}\to0,
\]
and
\begin{equation}\label{eq:automatic-fuglede-expansion}
\begin{aligned}
 P(F_j)-4\pi
 &=
 \frac12\int_{S^2}
 \bigl(|\nabla_{S^2}u_j|^2-2u_j^2\bigr)\,d\sigma+R_j, \\
 |R_j|
 &\le
 C\sqrt{\lambda_j}\int_{S^2}
        \bigl(u_j^2+|\nabla_{S^2}u_j|^2\bigr)\,d\sigma.
\end{aligned}
\end{equation}
\end{corollary}

\begin{proof}
For $j$ large, $\lambda_j<1/2$. Since $F_j\to B$ in Hausdorff
distance, Lemma~\ref{lem:hausdorff-support-radial} gives
$B_{1-\lambda_j}\subset F_j\subset B_{1+\lambda_j}$; in particular
$0\in\mathrm{int}(F_j)$. The slope bound then follows from
Lemma~\ref{lem:automatic-slope} applied to $K=F_j$, and the expansion
follows from Lemma~\ref{lem:fuglede-small-slope} with
$\eta_j=O(\sqrt{\lambda_j})$.
\end{proof}

\begin{corollary}
\label{cor:remainder-Q-bounded}
Let \((F_j)_j\subset\cA\) satisfy \(F_j\to B\) in Hausdorff
distance and \(\cQstar(F_j)\le M\) for every \(j\). Set
\[
        \lambda_j:=d_H(F_j,B),
        \qquad
        \delta_j:=\delta(F_j).
\]
Then
\[
        \delta_j\le \frac{C\lambda_j^2}{|\log\lambda_j|}
\]
for \(j\) large, and \eqref{eq:automatic-fuglede-expansion} satisfies
\[
        |R_j|\le C\lambda_j^{5/2}
        =
        o\!\left(\frac{\lambda_j^2}{|\log\lambda_j|}\right).
\]
In particular, along every \(\cQstar\)-bounded sequence in
\(\cA\) converging to \(B\), the perimeter expansion holds with
an \(o(\lambda_j^2/|\log\lambda_j|)\) remainder.
\end{corollary}

\begin{proof}
Since \(F_j\to B\) in the Hausdorff metric and the sets \(F_j\) are convex,
the perimeter is continuous along this convergence. Hence
$
        P(F_j)\to P(B)=4\pi,
        \delta_j\to0.
$
From \(\cQstar(F_j)\le M\) and
\(\lambda_*(F_j)\le \lambda_j=d_H(F_j,B)\), we get
\[
        \delta_j\log(\delta_j+\delta_j^{-1})
        =
        \frac{\cQstar(F_j)\lambda_*(F_j)^2}{P(F_j)^3}
        \le
        C\lambda_j^2.
\]
Since \(\delta_j\to0\),
$
        \log(\delta_j+\delta_j^{-1})
        =
        (1+o(1))\log\frac1{\delta_j}.
$
Thus, for \(j\) large,
$        \delta_j\log\frac1{\delta_j}
        \le
        C\lambda_j^2.
$
We claim that this implies
\[
        \delta_j\le C\frac{\lambda_j^2}{|\log\lambda_j|}.
\]
Indeed, if \(\delta_j>\lambda_j\) along a subsequence, then, for \(j\)
large enough so that \(\delta_j,\lambda_j<e^{-1}\), the monotonicity of
\(t\mapsto t\log(1/t)\) on \((0,e^{-1})\) gives
$
        \delta_j\log\frac1{\delta_j}
        \ge
        \lambda_j\log\frac1{\lambda_j}.
$
This contradicts
$
        \delta_j\log\frac1{\delta_j}
        \le
        C\lambda_j^2,
$
because
$
        \lambda_j\log\frac1{\lambda_j}\gg \lambda_j^2.
$
Hence \(\delta_j\le\lambda_j\) for \(j\) large. Consequently
\[
        \log\frac1{\delta_j}\ge |\log\lambda_j|,
\]
and the desired bound follows.
Set
\[
        A_j:=\int_{S^2}|\nabla_{S^2}u_j|^2\,d\sigma,
        \qquad
        B_j:=\int_{S^2}u_j^2\,d\sigma.
\]
By Lemma~\ref{lem:hausdorff-support-radial},
$
        B_j\le 4\pi\lambda_j^2.
$
By \eqref{eq:automatic-fuglede-expansion}, we obtain
\[
        P(F_j)-4\pi
        =
        \frac12(A_j-2B_j)+R_j.
\]
Therefore
\[
        A_j
        =
        8\pi\delta_j+2B_j-2R_j
        \le
        8\pi\delta_j+8\pi\lambda_j^2+2|R_j|.
\]
Since \(\delta_j=o(\lambda_j^2)\), this gives
$
        A_j+B_j\le C\lambda_j^2+2|R_j|.
$
By Corollary~\ref{cor:fuglede-automatic},
$
        |R_j|
        \le
        C\sqrt{\lambda_j}(A_j+B_j).
$
Thus
$
        |R_j|
        \le
        C\lambda_j^{5/2}
        +
        C\sqrt{\lambda_j}|R_j|.
$
For \(j\) large, the last term is absorbed into the left-hand side, and hence
$
        |R_j|\le C\lambda_j^{5/2}.
$
Since
$
        \sqrt{\lambda_j}\,|\log\lambda_j|\to0,
$
we have
$
        \lambda_j^{5/2}
        =
        o\!\left(\frac{\lambda_j^2}{|\log\lambda_j|}\right).
$
\end{proof}

\begin{remark}
\label{rem:no-non-small-slope}
The same argument applies to any translated representative
\(K_j=F_j-z_j\to B\). Indeed, \(P\), \(\delta\), volume, and \(\cQstar\)
are translation invariant, and the only place where the distance enters the
proof is through the bound
\[
        \lambda_*(F_j)\le d_H(K_j,B).
\]
Thus the conclusion holds with \(\lambda_j=d_H(K_j,B)\) and with the
remainder associated with the radial graph of \(K_j\).
\end{remark}

We are now ready to prove Theorem~\ref{thm:rough-local-gap}.
\begin{proof}
It is enough to prove the estimate along an arbitrary subsequence realizing
the lower limit. If \(\cQstar(F_j)\to+\infty\) along that subsequence, there is
nothing to prove. Hence, after passing to a further subsequence, we may assume
that
\[
        \sup_j \cQstar(F_j)\le M<+\infty.
\]

We now fix the center used for the radial parametrization. First, we choose a
translation for which the Hausdorff distance from the unit ball is almost
minimal, that is, close to \(\lambda_*\). We then make a further controlled
translation in order to impose the first-moment normalization
\[
        \int_{S^2} u\,\omega\,d\sigma=0,
\]
which is needed in the spectral estimate.

Choose translations \(y_j\in\mathbb R^3\) such that
\[
        \varepsilon_j:=d_H(F_j-y_j,B)\le 2\lambda_*(F_j).
\]
Since \(\lambda_*(F_j)\to0\), we have \(\varepsilon_j\to0\). Applying
Lemma~\ref{lem:first-moment-recentering} to
$
        \widehat K_j:=F_j-y_j,
$
we find vectors \(a_j\in\mathbb R^3\), with
$
        |a_j|\le C\varepsilon_j,
$
such that, setting
\[
        K_j:=F_j-y_j-a_j,
\]
and writing
$
        \partial K_j=\{(1+u_j(\omega))\omega:\omega\in S^2\},
$
we have the exact first-moment cancellation
\[
        \int_{S^2}u_j(\omega)\omega\,d\sigma(\omega)=0.
\]
Now we define
$
        \lambda_j:=d_H(K_j,B).
$
By the triangle inequality and the estimate on \(a_j\),
\[
        \lambda_j
        \le
        d_H(F_j-y_j,B)+|a_j|
        \le
        C\varepsilon_j
        \le
        C\lambda_*(F_j).
\]
On the other hand, by the definition of \(\lambda_*(F_j)\), applied to the
specific center \(y_j+a_j\),
\[
        \lambda_*(F_j)
        \le
        d_H(F_j,B+y_j+a_j)
        =
        d_H(K_j,B)
        =
        \lambda_j.
\]
Consequently,
$
        \lambda_*(F_j)\le \lambda_j\le C\lambda_*(F_j).
$
In particular, \(\lambda_j\to0\). Moreover, by
Lemma~\ref{lem:hausdorff-support-radial},
$
        \|u_j\|_{L^\infty(S^2)}=\lambda_j
$.
Then \(K_j\) is parametrized radially with respect to the origin. Equivalently,
the original set \(F_j\) is parametrized radially with respect to the center
\(y_j+a_j\).
Since
$
        \lambda_*(F_j)\le \lambda_j,
$
we have
\[
        \cQstar(F_j)
        =
        \frac{
        \delta(F_j)\log(\delta(F_j)+\delta(F_j)^{-1})P(F_j)^3
        }{\lambda_*(F_j)^2}
        \ge
        \frac{
        \delta(F_j)\log(\delta(F_j)+\delta(F_j)^{-1})P(F_j)^3
        }{\lambda_j^2}.
\]
Thus it is enough to obtain a lower bound for the right-hand side.
We set
\[
        L_j:=\frac{\lambda_j^2}{|\log\lambda_j|}.
\]
We apply Corollary~\ref{cor:remainder-Q-bounded} to the translated
representatives \(K_j\). Since the perimeter, the volume, the deficit and \(\cQstar\)
are translation invariant, the estimates apply with
\(\lambda_j=d_H(K_j,B)\). Hence
\begin{equation}\label{eq:def-delta-j-bound-Lj}
        \delta_j:=\delta(K_j)=\delta(F_j)
        \le
        C L_j,
\end{equation}
\begin{equation}\label{eq:perimeter-expansion-o-Lj}
        P(K_j)-4\pi
        =
        \frac12
        \int_{S^2}
        \bigl(|\nabla_{S^2}u_j|^2-2u_j^2\bigr)\,d\sigma
        +
        o(L_j).
\end{equation}
Since \(P(K_j)=P(F_j)\), this is also the perimeter expansion for \(F_j\).
We fix \(\alpha\in(0,1)\). By
Corollary~\ref{cor:one-propagated-disk}, there exist points
\(p_j\in S^2\), signs \(\sigma_j\in\{-1,1\}\), a constant
\(c_\alpha>0\), and a sequence \(\eta_j\to 0\), such that
\[
        \sigma_j u_j(\omega)
        \ge
        \alpha\lambda_j(1-\eta_j)
        \qquad
        \text{for every }\omega\in D_{c_\alpha\lambda_j}(p_j).
\]
We set
\[
        v_j:=\sigma_j u_j,
        \qquad
        h_j:=\alpha\lambda_j(1-\eta_j),
        \qquad
        r_j:=c_\alpha\lambda_j.
\]
Then
$
        h_j=\alpha\lambda_j(1+o(1)),
        |\log r_j|=|\log\lambda_j|+O(1).
$
The lower bound on \(v_j\) holds pointwise on
\(D_{r_j}(p_j)\), hence \(d\sigma\)-a.e.
We define
\begin{equation}\label{eq:def-Aj-Bj}
        A_j:=\int_{S^2}|\nabla_{S^2}u_j|^2\,d\sigma,
        \qquad
        B_j:=\int_{S^2}u_j^2\,d\sigma.
\end{equation}
We also set
\begin{equation}\label{eq:def-Qj-muj}
        Q_j^{(2)}:=A_j-2B_j,
        \qquad
        \mu_j:=\frac{B_j}{L_j}.
\end{equation}

Before applying Lemma~\ref{lem:spectral}, we verify that its compactness
assumptions are satisfied. By Lemma~\ref{lem:hausdorff-support-radial},
$
        \|u_j\|_{L^\infty(S^2)}=\lambda_j.
$
Hence, using the definition of \(B_j\) in \eqref{eq:def-Aj-Bj}, we get
\begin{equation}\label{eq:Bj-O-lambda-square}
        B_j
        =
        \int_{S^2}u_j^2\,d\sigma
        \le
        4\pi \lambda_j^2.
\end{equation}
We now estimate \(A_j\). By Corollary~\ref{cor:fuglede-automatic},
applied to \(K_j\), we have
$
        P(K_j)-4\pi
        =
        \frac12(A_j-2B_j)+R_j,
$
with
$
        |R_j|
        \le
        C\sqrt{\lambda_j}(A_j+B_j).
$
Therefore
\[
        A_j
        =
        2\bigl(P(K_j)-4\pi\bigr)+2B_j-2R_j,
\]
and so
\[
        A_j+B_j
        \le
        C\bigl(P(K_j)-4\pi\bigr)
        +C B_j
        +C\sqrt{\lambda_j}(A_j+B_j).
\]
Since
$
        P(K_j)-4\pi=P(F_j)-4\pi=4\pi\delta_j
$
and, by \eqref{eq:def-delta-j-bound-Lj},
$
        \delta_j\le C L_j
        =
        C\frac{\lambda_j^2}{|\log\lambda_j|}
        \le C\lambda_j^2,
$
we obtain, using also \eqref{eq:Bj-O-lambda-square},
\[
        A_j+B_j
        \le
        C\lambda_j^2
        +
        C\sqrt{\lambda_j}(A_j+B_j).
\]
For \(j\) large enough, the last term can be absorbed into the left-hand
side. Hence
\begin{equation}\label{eq:Aj-Bj-O-lambda-square}
        A_j+B_j=O(\lambda_j^2).
\end{equation}
Together with
$
        \|u_j\|_{L^\infty(S^2)}=\lambda_j\to0,
$
this yields
$
        u_j\to0
   \text{in }W^{1,2}(S^2)\cap L^\infty(S^2).
$
The volume constraint holds because \(K_j\in\cA\), and the first-moment
condition was imposed by construction. Therefore Lemma~\ref{lem:spectral}
applies and gives
\begin{equation}\label{eq:spectral-Qj-lower-bound}
        Q_j^{(2)}
        \ge
        4B_j+o(B_j).
\end{equation}
We claim that $(\mu_j)_j$ is bounded. Suppose by contradiction that $\mu_j\to+\infty$, i.e.,
\[
        \frac{B_j}{L_j}\to+\infty.
\]
By \eqref{eq:spectral-Qj-lower-bound}, we have
$
        \frac{Q_j^{(2)}}{L_j}\to+\infty.
$
Using \eqref{eq:perimeter-expansion-o-Lj}, \eqref{eq:def-Aj-Bj} and \eqref{eq:def-Qj-muj} we have
\begin{equation}\label{eq:perimeter-Qj-o-Lj}
        P(K_j)-4\pi
        =
        \frac12 Q_j^{(2)}+o(L_j).
\end{equation}
and
$
        \frac{P(K_j)-4\pi}{L_j}\to+\infty.
$
Equivalently,
$
        \frac{\delta_j}{L_j}\to+\infty.
$
This contradicts \eqref{eq:def-delta-j-bound-Lj}. Hence this case cannot occur
along a \(\cQstar\)-bounded sequence.
We can therefore assume that \((\mu_j)_j\) is bounded. After
passing to a further subsequence, we may assume that
\[
        \mu_j\to\mu\ge0.
\]
Then
$
        B_j=\mu_j L_j=O(L_j)=o(\lambda_j^2).
$
Since
$
        \|v_j\|_{L^2(S^2)}^2
        =
        \|u_j\|_{L^2(S^2)}^2
        =
        B_j,
$
we have
\[
        \|v_j\|_{L^2(S^2)}=o(\lambda_j)=o(h_j).
\]
Thus Lemma~\ref{lem:one-disk-logarithmic-estimate} applies to \(v_j\) on
the disk \(D_{r_j}(p_j)\), and gives
\[
        A_j
        =
        \int_{S^2}|\nabla_{S^2}v_j|^2\,d\sigma
        \ge
        2\pi\alpha^2 L_j(1+o(1)).
\]
Consequently,
\[
        \frac{Q_j^{(2)}}{L_j}
        =
        \frac{A_j-2B_j}{L_j}
        \ge
        2\pi\alpha^2-2\mu_j+o(1).
\]
On the other hand, by Lemma~\ref{lem:spectral}, we get
\begin{equation}\label{eq:spectral-Qj-over-Lj}
        \frac{Q_j^{(2)}}{L_j}
        \ge
        4\mu_j+o(1).
\end{equation}
The last two estimates imply
\[
        \frac{Q_j^{(2)}}{L_j}
        \ge
        \max\{2\pi\alpha^2-2\mu_j,\;4\mu_j\}+o(1).
\]
Taking the lower limit and using \(\mu_j\to\mu\), we get
$
        \liminf_{j\to\infty}\frac{Q_j^{(2)}}{L_j}
        \ge
        \max\{2\pi\alpha^2-2\mu,\;4\mu\}.
$

Since the above argument applies to every subsequence along which
\(\mu_j\) converges, we obtain
\[
        \liminf_{j\to\infty}\frac{Q_j^{(2)}}{L_j}
        \ge
        \min_{\mu'\ge0}
        \max\{2\pi\alpha^2-2\mu',\;4\mu'\}=\frac{4\pi\alpha^2}{3}.
\]
Using \eqref{eq:perimeter-Qj-o-Lj} we get
\[
        P(K_j)-4\pi
        \ge
        \left(\frac{2\pi\alpha^2}{3}+o(1)\right)L_j.
\]
Therefore
\begin{equation}\label{eq:delta-j-lower-bound-alpha}
        \delta_j
        =
        \frac{P(K_j)-4\pi}{4\pi}
        \ge
        \left(\frac{\alpha^2}{6}+o(1)\right)
        \frac{\lambda_j^2}{|\log\lambda_j|}.
\end{equation}
By
Corollary~\ref{cor:remainder-Q-bounded}, we also have
\begin{equation}\label{eq:delta-j-upper-bound-log}
        \delta_j
        \le
        C\frac{\lambda_j^2}{|\log\lambda_j|}.
\end{equation}
Combining \eqref{eq:delta-j-lower-bound-alpha} and \eqref{eq:delta-j-upper-bound-log}, we obtain
\[
        \log\frac1{\delta_j}
        \ge
        2|\log\lambda_j|+\log|\log\lambda_j|-O(1)
        =
        2|\log\lambda_j|+o(|\log\lambda_j|).
\]
Since \(\delta_j\to0\), we have
$
        \log(\delta_j+\delta_j^{-1})
        =
        \log\frac1{\delta_j}+o(1),
$
and therefore
$
        \log(\delta_j+\delta_j^{-1})
        \ge
        2|\log\lambda_j|+o(|\log\lambda_j|).
$
Combining this with \eqref{eq:delta-j-lower-bound-alpha}, we get
\[
        \delta_j\log(\delta_j+\delta_j^{-1})
        \ge
        \left(\frac{\alpha^2}{3}+o(1)\right)\lambda_j^2.
\]
Finally, since \(\delta_j\to0\), we have \(P(F_j)=P(K_j)\to4\pi\). Moreover,
\(\lambda_*(F_j)\le\lambda_j\). Hence
\[
        \cQstar(F_j)
\ge
        \frac{
        \delta_j\log(\delta_j+\delta_j^{-1})P(F_j)^3
        }{\lambda_j^2}                                      \ge
        \left(\frac{\alpha^2}{3}+o(1)\right)(4\pi)^3.
\]
Taking the lower limit and then letting \(\alpha\to 1\), we conclude that
\[
        \liminf_{j\to\infty}\cQstar(F_j)
        \ge
        \frac13(4\pi)^3
        =
        \frac{64}{3}\pi^3.
\]
This completes the proof.
\end{proof}

\section*{Acknowledgements}

This work is partially supported by the ANR project STOIQUES
financed by the
French Agence Nationale de la Recherche (ANR).
GP is a member of the Gruppo Nazionale per l'Analisi Matematica, la Probabilit\`a e le loro Applicazioni (GNAMPA) of the Istituto Nazionale di Alta Matematica (INdAM) and wishes to acknowledge financial support from INdAM GNAMPA Project 2026 ``Problemi di ottimizzazione di forma in contesti anisotropi e non-locali'' (CUP E53C25002010001). The work of GP was partially supported by the project ``Harmony'' (code MATE.\ IP.\ DR111.\ 2024.\ HARMONY) within the program of the University ``Luigi Vanvitelli'' and by the Portuguese government through FCT -- Funda\c{c}\~ao para a Ci\^encia e a Tecnologia, I.P., project 2024.14494.PEX with DOI identifier 10.54499/2024.14494.PEX (project ASSO).
The second author thanks D. Bucur, J. Lamboley and E. Parini for some useful discussions about this problem.

\end{document}